\documentclass[12pt, psamsfonts]{amsart}
\usepackage{amssymb,amsfonts,amsthm,amsmath,epsfig,hhline}
\usepackage{verbatim, enumerate}    %added

\usepackage[arc,all,graph,frame]{xy}

\theoremstyle{plain}
\newtheorem{theorem}{Theorem}
\newtheorem*{eg}{Erd\H{o}s--Gallai Theorem}
\newtheorem{corollary}{Corollary}

\theoremstyle{definition}

\newtheorem{example}{Example}
\newtheorem{remark}{Remark}

\def\d{\underbar{\em d}}

\title{An improvement of a result of  {Z}verovich--{Z}verovich} 

\author{Grant Cairns}
\author{Stacey Mendan}

\address{Department of Mathematics and Statistics, La Trobe University, Melbourne, Australia 3086}
\email{G.Cairns@latrobe.edu.au}
\email{spmendan@students.latrobe.edu.au}

\begin{document}

\maketitle

\begin{abstract}
We give an improvement of a result of Zverovich and Zverovich which gives a condition on the first and last elements in a decreasing sequence of positive integers for the sequence to be graphic, that is, the degree sequence of a finite graph.
\end{abstract}

\section{Statement of Results}

A finite sequence of positive  integers is \emph{graphic} if it occurs as the sequence of vertex degrees of a  graph. Here,  graphs are understood to be \emph{simple}, in that they have no loops or repeated edges.
A result of Zverovich and Zverovich states:

\begin{theorem}[{\cite[Theorem 6]{ZZ}}]\label{T:origzz}
Let $a,b$ be reals. If $\emph{\d}=(d_1,\dots,d_n)$ is a sequence of positive integers in decreasing order with $d_1\leq a, d_n\geq b$ and
\[
n\geq \frac{(1+a+b)^2}{4b},
\]
then $\emph{\d}$  is graphic.
\end{theorem}

Notice that here the term $\frac{(1+a+b)^2}{4b}$ is monotonic increasing in $a$, for $a\geq 1$ and fixed $b$, and it is also monotonic decreasing in $b$, for $a\geq b\geq 1$ and fixed $a$. Thus any sequence that satisfies the inequality $n\geq \frac{(1+a+b)^2}{4b}$, for any pair $a\geq d_1, b\leq d_n$, will also satisfy the inequality  $n\geq  \frac{(1 + d_1 +  d_n)^2}{4  d_n }$.  So Theorem \ref{T:origzz} has the following equivalent expression.

\begin{theorem}\label{T:zz}
Suppose that $\emph{\d}=(d_1,\dots ,d_n)$ is a decreasing sequence  of positive integers with even sum. If  
\begin{equation}\label{E:zz}
n\geq  \frac{(1 + d_1 +  d_n)^2}{4  d_n },
\end{equation}
then $\emph{\d}$  is graphic.
\end{theorem}

  The simplified form of Theorem  \ref{T:zz} also affords a somewhat simpler proof, which we give in Section \ref{S:alons}
below. Admittedly, the proof in \cite{ZZ} is already quite elementary, though it does use the strong index results of \cite{Li,HIS}.

The following corollary of Zverovich--Zverovich's is obtained by taking $a=d_1$ and $b=1$ in Theorem \ref{T:origzz}.

\begin{corollary}[{\cite[Corollary 2]{ZZ}}]\label{C:zz}
Suppose that $\emph{\d} =(d_1,\dots ,d_n)$ is a decreasing sequence  of positive integers with even sum. If  $n\geq \frac{d_1^2}{4}+d_1+1$, then $\emph{\d}$  is graphic.
\end{corollary}

Zverovich--Zverovich state that the bound of Corollary \ref{C:zz} ``cannot be improved'', and they give examples to this effect. In fact, there is an improvement, as we will now describe. The subtlety here is that in the Zverovich--Zverovich  examples, for a given (integer) value of $n$, the Corollary \ref{C:zz} bound can't be improved for integer $d_1$. Nevertheless the bound on $n$, for given integer $d_1$, can be improved. We prove the following result in Section \ref{S:alons}.

\begin{theorem}\label{T:ourzz}
Suppose that $\emph{\d} =(d_1,\dots ,d_n)$ is a decreasing sequence  of positive integers with even sum. If  $n\geq \left\lfloor\frac{d_1^2}{4}+d_1\right\rfloor$, then $\emph{\d}$  is graphic.
\end{theorem}

\begin{example}
There are many examples of sequences that verify the hypotheses of Theorem \ref{T:ourzz} but not those of Corollary \ref{C:zz}. For example, for every positive odd integer $x$, consider the sequence $(2x,1^{x^2+2x-1})$, and for $x$ even,   consider the sequence $(2x,2x,1^{x^2+2x-2})$. Here, and in sequences throughout this paper, the superscripts indicate the number of repetitions of the entry. \end{example}

\begin{example}\label{Eg}
The following examples show that the bound of Theorem \ref{T:ourzz} is sharp. For $d$ even, say $d=2x$ with $x\geq 1$, let 
$\d =(d^{x+1},1^{x^2+x-2})$.
For $d$ odd, say $d=2x+1$  with $x\geq 1$, let 
$\d=(d^{x+1},1^{x^2+2x-1})$.
In each case  $g$ has even sum, $n=\left\lfloor\frac{d^2}{4}+d\right\rfloor-1$, but $\d$ is not graphic, as one can see from the Erd\H{o}s--Gallai Theorem  \cite{TVW}.
% with $k=x+1$.

\begin{comment}
Even case
 $d(x+1)=2x(x+1)>(x+1)x+  x^2+x-2$.

Odd case:
$n=x^2+3x$ while $\frac{d^2}{4}+d-1= x^2+3x+\frac{1}{4}$
 $d(x+1)=(2x+1)(x+1)>(x+1)x+  x^2+2x-1$
\end{comment}
\end{example}

\begin{remark}
The fact that Theorem \ref{T:zz} is not sharp has also been remarked in \cite{BHJW}, in the abstract of which the authors state that Theorem \ref{T:zz}  is ``sharp  within 1''. They give the bound 
\begin{equation}\label{E:BHJW}
n\geq\frac{(1 + d_1 +  d_n)^2-\epsilon'}{4  d_n },
\end{equation}
 where $\epsilon'=0$ if $d_1 +  d_n$ is odd, and $\epsilon'=1$ otherwise.
Consider any decreasing sequence with $d_1=2x+1$ and $d_n=1$. Note that the bound given by Theorem \ref{T:zz} is $n\geq x^2+3x+3$,
the bound given by \eqref{E:BHJW} is $n\geq x^2+3x+2$, while Theorem \ref{T:ourzz} gives the stronger bound  $n\geq x^2+3x+1$.  The paper \cite{BHJW} gives more precise bounds, as a function of $d_1,d_n$, and the maximal gap in the sequence.
\end{remark}

\begin{remark}
There are many  other recent papers on graphic sequences; see for example \cite{TT,Yin,BHJW,CSloops}.
\end{remark}

\section{Proofs of Theorems  \ref{T:zz}  and \ref{T:ourzz}}\label{S:alons}

We will require the Erd\H{o}s--Gallai Theorem, which we recall for convenience.

\begin{eg}A sequence $\emph{\d}=(d_1 ,\dots, d_n )$ of nonnegative integers in decreasing order is graphic if and only if its sum is even and, for each integer $k$ with $1 \leq  k \leq  n$,
\begin{equation*}
\sum_{i=1}^k d_i\leq  k(k-1)+ \sum_{i=k+1}^n\min\{k,d_i\}.\tag{EG}
\end{equation*}
\end{eg}

\begin{proof}[Proof of Theorem \ref{T:zz}]
Suppose that $\d=(d_1,\dots ,d_n)$ is a decreasing sequence with even sum, satisfying \eqref{E:zz}, and which is not graphic. By the Erd\H{o}s--Gallai Theorem, there exists $k$ with $1 \leq  k \leq  n$, such that
\begin{equation}\label{E:notEG}
\sum_{i=1}^k d_i>  k(k-1)+ \sum_{i=k+1}^n\min\{k,d_i\}.
\end{equation}
For each $i$ with $1\leq i\leq k$, replace $d_i$ by $d_1$; the left hand side of \eqref{E:notEG} is not decreased, while the right hand side of \eqref{E:notEG} is unchanged, so \eqref{E:notEG} still holds. 
%Since the left hand side of \eqref{E:notEG} is now $kd_1$,  we have $kd_1 >  k(k-1) $ and so $d_1\geq k$.  
Now for each $i$ with $k+1\leq i\leq n$, replace $d_i$ by $d_n$;  the left hand side of \eqref{E:notEG} is unchanged, while the right hand side of \eqref{E:notEG} has not increased, so \eqref{E:notEG} again holds. Notice that if $k< d_n$, then \eqref{E:notEG} gives $kd_1>  k(k-1)+(n-k)k=k(n-1)$, and so $d_1\geq n$. Then \eqref{E:zz} would give $4n   d_n \geq  (1 + d_n + n)^2$, that is, $(n-(d_n-1))^2-(d_n-1)^2+(1 + d_n)^2\leq 0$. But  this inequality clearly has no solutions. Hence $k\geq d_n$. Thus \eqref{E:notEG} now reads 
$kd_1>  k(k-1)+(n-k)d_n$, or equivalently
\[
(k-\frac12(1+d_1+d_n))^2- \frac14(1+d_1+d_n)^2 +nd_n<0.
\]
But this contradicts the hypothesis.
\end{proof}

The following proof uses the same general strategy as the preceding proof, but  requires a somewhat more careful argument.

\begin{proof}[Proof of Theorem \ref{T:ourzz}]

Suppose that $\d$ satisfies the hypotheses of the theorem.  First suppose that $d_1$ is even, say $d_1=2x$. If $d_n\geq 2$, then
since $\frac{(1 + d_n +  d_1)^2}{4  d_n }$  is a strictly monotonic decreasing function of $d_n$ for $1\leq d_n\leq d_1$, we have
\[
n\geq  \frac{d_1^2}{4}+d_1= \frac{(2+d_1)^2}{4}-1> \frac{(1 + d_n +  d_1)^2}{4  d_n }-1,
\]
so $n\geq  \frac{(1 + d_n +  d_1)^2}{4  d_n }$ and hence $\d$ is graphic by Theorem \ref{T:zz}.
So, assuming that $\d$ is not graphic, we may suppose that $d_n=1$. Furthermore, by Corollary \ref{C:zz}, we may assume that $n=\frac{d_1^2}{4}+d_1$, so $n=x^2+2x$.

Now, as in the proof of Theorem \ref{T:zz}, by the Erd\H{o}s--Gallai Theorem, there exists $k$ with $1 \leq  k \leq  n$, such that
\begin{equation}\label{E:notEGagain}
\sum_{i=1}^k d_i>  k(k-1)+ \sum_{i=k+1}^n\min\{k,d_i\}.
\end{equation}
For each $i$ with $1\leq i\leq k$, replace $d_i$ by $d_1$; the left hand side of \eqref{E:notEGagain} is not decreased, while the right hand side of \eqref{E:notEGagain} is unchanged, so \eqref{E:notEGagain} still holds. 
For each $i$ with $k+1\leq i\leq n$, replace $d_i$ by $1$;  the left hand side of \eqref{E:notEGagain} is unchanged, while the right hand side of \eqref{E:notEGagain} has not increased, so \eqref{E:notEGagain} again holds.
Then \eqref{E:notEGagain} reads
$kd_1>  k(k-1)+(n-k)$, and consequently, rearranging terms, $ (k-x-1)^2-1<0$. 
Thus $k=x+1$. Notice that for $1\leq i\leq k$, if any of the original terms $d_i$  had been  less than $d_1$, we would have obtained $ (k-x-1)^2<0$, which is impossible. Similarly, for $k+1\leq i\leq n$, all the original terms $d_i$ must have been all equal to one. Thus $\d=(d_1^{k},1^{n-k})=((2x)^{x+1},1^{x^2+x-1})$. So $\d$ has sum  $2x(x+1)+x^2+x-1= 3x^2 +3x-1$, which is odd, regardless of whether $x$ is even or odd. This contradicts the hypothesis.

Now consider the case where $d_1$ is odd, say $d_1=2x-1$. The theorem is trivial for $\d=(1^n)$, so we may assume that $x>1$. We use essentially the same approach as we used in the even case, but the odd case is somewhat more complicated. 
By Corollary \ref{C:zz}, assuming $\d$ is not graphic, we have $\frac{d_1^2}{4}+d_1+1>n$, and hence, as $d_1$ is odd, $\frac{d_1^2}{4}+d_1+\frac34\geq n$.
Thus, since $n\geq  \left\lfloor\frac{d_1^2}{4}+d_1\right\rfloor=\frac{d_1^2}{4}+d_1-\frac14$, we have 
$n=\frac{d_1^2}{4}+d_1+\frac34$ or $n=\frac{d_1^2}{4}+d_1-\frac14$. Thus there are two cases:
\begin{enumerate}
\item[(i)] $n=x^2+x-1$,
\item[(ii)] $n=x^2+x$.
\end{enumerate}
By the Erd\H{o}s--Gallai Theorem, there exists $k$ with $1 \leq  k \leq  n$, such that
\begin{equation}\label{E:notEGagain2}
\sum_{i=1}^k d_i>  k(k-1)+ \sum_{i=k+1}^n\min\{k,d_i\}.
\end{equation}
As before, for each $i$ with $1\leq i\leq k$, replace $d_i$ by $d_1$ and for each $i$ with $k+1\leq i\leq n$, replace $d_i$ by $d_n$, and note that \eqref{E:notEGagain2} again holds.
Arguing as in the proof of Theorem \ref{T:zz}, notice that if $k< d_n$, then \eqref{E:notEGagain2} gives $kd_1>  k(k-1)+(n-k)k=k(n-1)$, and so $d_1\geq n$. In both cases (i) and (ii) we would have $2x-1\geq n\geq x^2+x-1$ and hence $x\leq1$, contrary to our assumption. Thus $k\geq d_n$ and  \eqref{E:notEGagain2} reads
$kd_1>  k(k-1)+(n-k)d_n$, and consequently, rearranging terms, we obtain in the respective cases:
\begin{enumerate}
\item[(i)] $d_n x^2  - d_n k + k^2 + d_n x - 2 k x -d_n<0$.
\item[(ii)] $d_n x^2 -d_n k + k^2 + d_n x - 2 k x <0$,
\end{enumerate}
In both cases we have $d_n x^2  - d_n k + k^2 + d_n x - 2 k x -d_n<0$. Consider  $d_n x^2  - d_n k + k^2 + d_n x - 2 k x -d_n$ as a quadratic in $k$. For this to be negative, its discriminant, $4 d_n + d_n^2 + 4 x^2 - 4 d_n x^2$, must be positive. If $d_n>1$ we obtain  $x^2< \frac{4 d_n + d_n^2}{4(d_n-1)}$. 
For $d_n=2$ we have $x^2< 3$ and so $x=1$, contrary to our assumption.
Similarly, for $d_n=3$ we have $x^2< \frac{21}8$ and so again $x=1$.
For $d_n\geq 4$, the function $\frac{4 d_n + d_n^2}{4(d_n-1)}$ is monotonic increasing in $d_n$. So, as $d_n\leq d_1$,
\[
x^2< \frac{4 d_1 + d_1^2}{4(d_1-1)}=\frac{4x^2+4x-3}{8x-8}<\frac{x^2+x}{2(x-1)},
\]
which again gives $x=1$. We conclude that  $d_n=1$.

So the two cases are:
\begin{enumerate}
\item[(i)] $ x^2  -  k + k^2 +  x - 2 k x -1=(k - x) (k - x - 1)-1 <0$.
\item[(ii)] $ x^2 - k + k^2 +  x - 2 k x=(k - x) (k - x - 1) <0$,
\end{enumerate}
In case (ii) we must have $x<k<x+1$, but this is impossible for integer $k$ and $x$.

In case (i), either $k=x$ or $k=x+1$. Notice that for $1\leq i\leq k$, if any of the original terms $d_i$  had been  less than $d_1$, we would have obtained $ (k - x) (k - x - 1) <0$, which is impossible. Similarly, for $k+1\leq i\leq n$, all the original terms $d_i$ must have been all equal to one. Thus 
$\d=(d_1^{k},1^{n-k})$. Consequently, if $k=x$, we have $\d=((2x-1)^{x},1^{x^2-1})$  as $n=x^2+x-1$. In this case, $\d$ has sum  $x(2x-1)+x^2-1= 3x^2-x -1$, which is odd, regardless of whether $x$ is even or odd, contradicting the hypothesis.
On the other hand, if $k=x+1$, we have $\d=((2x-1)^{x+1},1^{x^2-2})$. Here, $\d$ has sum  $(2x-1)(x+1)+x^2-2= 3x^2 +x-3$, which is again odd, regardless of whether $x$ is even or odd, contrary to the hypothesis.
\end{proof}

\bibliographystyle{amsplain}
\bibliography{erdosgallai}

  \end{document}